\documentclass[oneside,english]{amsart}
\usepackage[T1]{fontenc}
\usepackage[latin9]{inputenc}
\usepackage{amsthm}
\usepackage{amssymb}

\makeatletter
\numberwithin{equation}{section}
\numberwithin{figure}{section}
\theoremstyle{plain}
\newtheorem{thm}{\protect\theoremname}
  \theoremstyle{plain}
  \newtheorem{lem}[thm]{\protect\lemmaname}
  \theoremstyle{plain}
  \newtheorem{cor}[thm]{\protect\corollaryname}
  \theoremstyle{plain}
  \newtheorem{prop}[thm]{\protect\propositionname}
  \theoremstyle{remark}
  \newtheorem{rem}[thm]{\protect\remarkname}
  \theoremstyle{remark}
  \newtheorem{example}[thm]{\protect\examplename}

\usepackage{amsmath, amscd, amsthm, amssymb, graphics, xypic, mathrsfs, setspace, fancyhdr, times, bm}
\DeclareFontFamily{OT1}{pzc}{}
\DeclareFontShape{OT1}{pzc}{m}{it}%
             {<-> s * [1.195] pzcmi7t}{}
\DeclareMathAlphabet{\mathscr}{OT1}{pzc}%
                                 {m}{it}
\usepackage[colorlinks=true,pagebackref=true]{hyperref} % new hyperref package added by Jean
\hypersetup{backref}

\newcommand{\Spec}{\operatorname{Spec}}

\newcommand{\OO}{{\mathcal O}}

\newcommand{\Z}{{\mathbb Z}}
\newcommand{\N}{{\mathbb N}}
\newcommand{\A}{{\mathbb A}}

\newcommand{\aone}{{\mathbb A}^1}
\newcommand{\pone}{{\mathbb P}^1}

\newcommand{\KMW}{\mathbf{K}^{\mathrm{MW}}}

\makeatother

\usepackage{babel}
  \providecommand{\corollaryname}{Corollary}
  \providecommand{\lemmaname}{Lemma}
  \providecommand{\propositionname}{Proposition}
  \providecommand{\remarkname}{Remark}
\providecommand{\theoremname}{Theorem}
\providecommand{\examplename}{Example}

\begin{document}
\setlength{\headheight}{1cm}
\author{Adrien Dubouloz \and Jean Fasel}
\address{Adrien Dubouloz \\ IMB - UMR 5584, CNRS, Univ. Bourgogne Franche-Comté, F-21000 Dijon, France} 
\address{Jean Fasel \\ Institut Fourier - UMR 5582 \\ Universit\'e Grenoble Alpes \\ 100 rue des math\'ematiques\\ F-38000 Grenoble }
%\subjclass[2010]{Primary: 13A15, 13C40, 14M10; Secondary: 14R10, 19M05, 19G38}
\thanks{The first author was partially supported by ANR Grant \textquotedbl{}BirPol\textquotedbl{}
ANR-11-JS01-004-01.}
\title{{\bf Families of $\mathbb{A}^1$-contractible affine threefolds}} 
\date{} 
\maketitle

\begin{abstract} We provide families of affine threefolds which are contractible in the unstable $\mathbb A^1$-homotopy category of Morel-Voevodsky and pairwise non-isomorphic, thus answering Conjecture 1.4 in \cite{Asok07}. As a particular case, we show that the Koras-Russell threefolds of the first kind are contractible, extending results of \cite{Hoyois15}.
\end{abstract}

\section*{Introduction}
Let $k$ be a field and let $X$ be a smooth affine scheme of dimension $d$ over $k$. The Zariski cancellation problem asks if the existence of an isomorphism $X\times \A^1_k\simeq \A^{d+1}_k$ implies that $X$ is isomorphic to $\A^d_k$. This question is known to have a negative answer when $k$ has positive characteristic by work of T. Asanuma and N. Gupta (\cite{As87,Gupta14}) but is still open in characteristic zero for any  $d\geq 3$. Among possible candidate counter-examples are the so-called Koras-Russell threefolds of the first kind $X(m,r,s)$, which are the closed subvarieties of $\A^4_k$ defined by the equations $x^mz=y^r+t^s+x$ where  $m\geq 2$ and $r,s\geq 2$ are coprime integers. For more details on these threefolds, we refer to the nice introduction of \cite{Hoyois15}. All these threefolds admit algebraic actions of the  additive group $\mathbb{G}_{a,k}$ and they were originally proven to be not isomorphic to the affine space $\A^3_k$ by means of invariants associated to those actions \cite{KML97}. But one might expect that they are stably isomorphic to $\A^3_k$, i.e. that there exists an isomorphism $X(m,r,s)\times \A^1_k\simeq \A^3_k\times \A^1_k$. If such an isomorphism exists, then $X(m,r,s)$ is contractible in the (unstable) $\A^1$-homotopy category $\mathcal H(k)$ of Morel and Voevodsky (\cite{Morel99}) and it follows that the contractibility of $X$ is an obstruction to the existence of an isomorphism as above.

In a recent paper, Hoyois, Krishna and \O stv\ae r proved that the Koras-Russell threefolds are stably contractible, i.e. that there exists $n\geq 0$ such that $(\pone)^{\wedge n}\wedge X(m,r,s)$ is contractible. Basically, that means that these threefolds have no non-trivial cohomology for any cohomology theory which is representable in the stable category of $\pone$-spectra. Let us note however that examples of non isomorphic spaces in $\mathcal H(k)$ that become isomorphic after a single smash product with $\pone$ are abundant in nature. We refer the interested reader to \cite[Proposition 5.22]{AsokMorel} or to \cite[Theorem 4.2]{Wickelgren15} for such examples. The first theorem of the present paper shows that the Koras-Russell threefolds are indeed contractible.

\begin{thm}\label{thm:KRcontractible}
The Koras-Russell threefolds $X(m,r,s)=\{x^{m}z=y^{r}+t^{s}+x\}$ are $\mathbb{A}^{1}$-contractible. 
\end{thm} 

While the methods in \cite{Hoyois15} are quite sophisticated, such general techniques are not available in $\mathcal H(k)$ and our argument is more elementary. Yet, our proof requires a non trivial geometric construction which proves, together with an application of the Brouwer degree in motivic homotopy theory as developed in \cite[Theorem 6.40, Corollary 6.43]{Morel08}, that $L=\{x=y=z=0\}$ is an affine line $\A^1$ in $X(m,r,s)$ such that $X(m,r,s)\setminus L$ is weak-equivalent to $\A^2_k\setminus \{0\}$. The rest of the argument rests on a weak version of the five lemma (Lemma \ref{lem:vwfl} below) which works in the general framework of pointed model categories.  Let us note that the results of \cite{Hoyois15} are consequences of our theorem above. 

A more general cancellation problem is to know whether two smooth affine schemes $X$ and $Y$ over $k$ such that there exists an isomorphism $X\times\A^1_k\simeq Y\times \A^1_k$ are actually isomorphic. In general, the answer is known to be negative (see e.g. \cite{Ru14} for a survey), and we consider a set of examples in this paper, which generalizes families introduced in \cite{DMJ11} in the complex case. Namely, let $q(x)\in k[x]$ be a polynomial such that $q(0)\in k^\times$, and consider the closed subvariety $X(m,r,s,q)$ defined by the equation $x^{m}z=y^{r}+t^{s}+q(x)x$. 

\begin{thm}\label{thm:family}
For every fixed $m\geq2$, $r,s\geq2$ with $\gcd(r,s)=1$, the following
holds:
\begin{itemize}
\item[a)] The varieties $X(m,r,s,q_{1})$ and $X(m,r,s,q_{2})$ are $k$-isomorphic
if and only if there exists $\lambda,\varepsilon\in k^{*}$ such that
$q_{2}(x)\equiv\varepsilon q_{1}(\lambda x)$ modulo $x^{m-1}$.
\item[b)] The cylinders $X(m,r,s,q)\times\mathbb{A}^{1}_k$ are all isomorphic.
\end{itemize}
\end{thm} 

In particular, we see that the varieties $X(m,r,s,q)$ and $X(m,r,s,q(0))\cong X(m,r,s)$ are stably isomorphic and hence isomorphic in $\mathcal H(k)$. In view of Theorem \ref{thm:KRcontractible}, it follows that the varieties $X(m,r,s,q)$ are all contractible. Consequently, we obtain moduli of arbitrary positive dimension of pairwise non isomorphic, $\mathbb{A}^{1}$-stably
isomorphic, $\mathbb{A}^{1}$-contractible smooth affine threefolds.

\begin{cor}\label{cor:family}
Let $Y=\mathrm{Spec}(k[a_{2},\ldots,a_{m-1}])$, $m\geq4$
and let $\mathfrak{X}(m,r,s)\subset Y\times\mathbb{A}^{3}$ be the
subvariety defined by the equation 
\[
x^{m}z=y^{r}+t^{s}+x+x^{2}+\sum_{i=2}^{m-1}a_{i}x^{i+1}.
\]
Then $\mathrm{p}_{Y}:\mathfrak{X}(m,r,s)\rightarrow Y$ is a smooth
family whose fibers, closed or not, are all $\mathbb{A}^{1}$-contractible
and non isomorphic to $\mathbb{A}^{3}$ over the corresponding residue
fields. Furthermore, the fibers of $\mathrm{p}_{Y}$ over the $k$-rational
points of $Y$ are pairwise non isomorphic $k$-varieties which
are all $\mathbb{A}^{1}$-stably isomorphic.
\end{cor}

Let us try to put the above result into context. Inspired from the topological situation (see the beautiful introduction of \cite{Asok07} for a survey), Asok and Doran asked a series of interesting questions on $\aone$-contractible varieties. The first one asks if every smooth $\A^1$-contractible variety can be constructed as the quotient of an affine space by the free action of a unipotent group (\cite[Question 1.1]{Asok07}). While there are many examples of $\aone$-contractible varieties constructed in this manner, the answer to this question is now known to be negative in general (\cite[Theorem 3.1.1, Corollary 3.2.2]{Asok14}). The next question of interest is the abundance of "exotic" $\aone$-contractible varieties, i.e. $\aone$-contractible varieties that are not isomorphic to an affine space, especially affine ones. In \cite[Theorem 5.3]{Asok07}, the authors show that for any $n\geq 0$ there exists a connected $k$-scheme $S$ of dimension $n$ and a smooth morphism $f:Z\to S$ of arbitrary relative dimension $m\geq 6$ such that the fibers over rational points are pairwise non-isomorphic, quasi-affine and $\A^1$-contractible. More generally, they were able to prove that there are infinitely many isomorphism classes of $\aone$-contractible quasi-affine varieties of dimension $m\geq 4$ (\cite[Theorem 5.1]{Asok07}). Together with the fact that there are no exotic $\aone$-contractible surfaces, the picture was pretty much complete leaving only the case of threefolds open. Corollary \ref{cor:family} above fills this gap and answers in particular Conjecture 1.4 in \cite{Asok07}.

The organization of the paper is as follows. In Section \ref{sec:prelim}, we give a very quick introduction to the (unstable) $\aone$-homotopy category $\mathcal H(k)$ of Morel and Voevodsky focusing only on the features necessary to understand this paper. In Section \ref{sec:main}, we give the proofs of our theorems modulo some technical results that are deferred to Section \ref{sec:technical}. The point in separating the core of the arguments and the technicalities is to make the structure more transparent, and we hope we have succeeded in that task, at least to some extent. 

To conclude, let us mention that there is a second family of Koras-Russell threefolds $X$, defined by equations of the form $(x^m+t^s)^{d}z=y^r+x$ in $\A^4_k$, where $d\geq 1$, and $m,r,s \geq 2$ are integers such that $m$ and $rs$ are coprime \cite{KR97}. We expect that the results of this paper will also hold for this family: in particular, these threefolds also contain a special affine line $L=\{x=y=t=0\}$, but we don't know at the moment whether $X\setminus L$ is weakly equivalent to $\A^2\setminus \{0\}$.

\subsection*{Conventions}

We work over a base field $k$ of characteristic $0$. The schemes are essentially of finite type over $\Spec k$ and separated.

\subsection*{Acknowlegments}
We would like to thank Aravind Asok for very useful conversations and comments.

\section{Preliminaries}\label{sec:prelim}

\subsection{A user's guide to the $\aone$-homotopy category}

The rough idea of the construction is to enlarge the category of smooth (separated and finite type) $k$-schemes $\mathrm{Sm}_k$ in order to be able to perform quotients in general, add some simplicial information in order to be able to do homotopy theory and then artificially invert all the morphisms $X\times \A^1\to X$. We refer to \cite{Morel99} for more information and only sketch the necessary steps. If $X$ is a smooth scheme, then it can be considered as a sheaf of sets $X:\mathrm{Sm}_k\to \mathrm{Sets}$ in any reasonable Grotendieck topology. For reasons explained in \cite{Morel99}, the convenient topology to consider is the Nisnevich topology. Recall that covers $\{u_\alpha:X_\alpha\rightarrow X\}$ for this topology are collections of \'etale maps such that for every (possibly non-closed) point $x\in X$, there exists $\alpha$ and a point $y\in X_{\alpha}$ such that $u_{\alpha}(y)=x$ and the induced map of residue fields $k(x)\rightarrow k(y)$ is an isomorphism. Now, any set can be seen as a (constant) simplicial set, and it follows that $X$ can be seen as a sheaf of simplicial sets $X:\mathrm{Sm}_k\to \mathrm{SSets}$. On the other hand, any simplicial set can be seen as a constant sheaf of simplicial sets and it follows that both schemes and simplicial sets can be seen as objects of the category $\mathrm{Spc}_k$ of simplicial (Nisnevich) sheaves on $\mathrm{Sm}_k$. The objects of $\mathrm{Spc}_k$ are called \emph{spaces}. In particular, $\Spec k$ is a space, and a \emph{pointed space} is an object $\mathscr X$ of $\mathrm{Spc}_k$ together with a morphism $x:\Spec k\to \mathscr X$ that we often denote by $(\mathscr X,x)$ or even $\mathscr X$ if the base point is obvious. 

Any space $\mathscr X$ has stalks at the points of the Nisnevich topology, and a (pointed) morphism of spaces $f:\mathscr X\to \mathscr Y$ is said to be a weak-equivalence if it induces a weak-equivalence of simplicial sets on stalks. One can put a model structure on $\mathrm{Spc}_k$ that allow to invert weak-equivalences in a good way, and the corresponding homotopy category is the \emph{simplicial homotopy category of smooth schemes}. Further inverting the morphisms $\mathscr X\times \A^1\to \mathscr X$, one gets the $\A^1$-homotopy category of schemes $\mathcal H(k)$ and its pointed version $\mathcal H_\bullet(k)$. For simplicity, we often omit the base point in the notation. One of the principal feature of $\mathcal H_\bullet(k)$ is that it is a pointed model category. As such, there is a notion of cofiber sequence which can be seen as a machine to produce long exact sequences of pointed sets and groups. Another important feature is that the projection morphisms $X\times \A^1\to X$ are isomorphisms in $\mathcal H(k)$, and more generally that the projection of an affine bundle $Y\to X$ to its base is an isomorphism. This is in particular the case for torsors under vector bundles. Finally, let us recall that a space is called \emph{contractible} if the natural morphism $\mathscr X\to \Spec k$ is an isomorphism in $\mathcal H(k)$.  

\section{Proofs}\label{sec:main}

\subsection{Proof of Theorem \ref{thm:KRcontractible}}

Let then $X(m,r,s)=\{x^mz=y^r+t^s+x\}$ be as in the introduction, and write $X=X(m,r,s)$ for simplicity. Every such $X$ contains a hypersurface $P=\{z=0\}\subset X$ isomorphic to $\mathbb{A}^{2}=\mathrm{Spec}(k[y,t])$, and it is therefore enough to show that
the inclusion $P\hookrightarrow X$ is a weak-equivalence. On the
other hand, we observe that $X$ contains a line $L=\{x=y=t=0\}\cong\mathrm{Spec}(k[z])\cong\mathbb{A}^{1}$
intersecting $P$ transversally at the unique point $(0,0,0,0)$,
for which we have a pull-back square 
\[ 
\xymatrix{P\setminus L\cong \A^2\setminus \{0\}\ar[r]\ar[d]_-i & P\cong\A^2\ar[d] \\ X\setminus L\ar[r] & X} 
\] 
and an associated commutative diagram of cofiber sequences

\[\xymatrix{\A^2\setminus \{0\}\ar[r]\ar[d]_-i & \A^2\ar[d]\ar[r] & (\pone)^{\wedge 2}\ar[d] \\ X\setminus L\ar[r] & X\ar[r] & X/(X\setminus L).}\]

Applying a variant of the five lemma described in Section \ref{sub:weak} below, we are reduced to prove that
 $i:\mathbb{A}^{2}\setminus \{0\}\cong P\setminus L\hookrightarrow X\setminus L$
and the induced map $(\pone)^{\wedge 2}\rightarrow X/(X\setminus L)$
are both weak-equivalences. This is done in two steps: we first construct
in subsection \ref{sub:weak-equiv} below an explicit weak-equivalence
$X\setminus L\simeq\mathbb{A}^{2}\setminus \{0\}$. Together with the fact
that the normal cone to $L$ in $X$ is generated by the global sections
$y$ and $t$, this implies in particular that $(\pone)^{\wedge 2}\rightarrow X/(X\setminus Z)\simeq (\pone)^{\wedge 2}$
is a weak-equivalence. Then it is enough to show that the composition
$\mathbb{A}^{2}\setminus \{0\}\cong P\setminus L\hookrightarrow X\setminus L\simeq\mathbb{A}^{2}\setminus \{0\}$
is also weak-equivalence, which is done by explicit computation in
subsection \ref{sub:computations} below.

\subsection{Proof of Theorem \ref{thm:family}}

Let us now consider $X(m,r,s,q)$ as in the introduction. Recall that these varieties are defined
by equations of the form 
\[
x^{m}z=y^{r}+t^{s}+xq(x)
\]
where $m\geq2$, $r,s\geq1$ are relatively prime and $q\in k[x]$
is a polynomial such that $q(0)\in k^{*}$. Note that if either $r$
or $s$ is equal to $1$ then $X(m,r,s,q)\cong \A^3=\mathrm{Spec}(k[x,t,z])$ or  
$X(m,r,s,q)\cong \A^3=\mathrm{Spec}(k[x,y,z])$ respectively. Otherwise, if $r,s\geq2$,
then the restriction $\pi:X(m,r,s,q)\rightarrow\mathbb{A}^{1}$
of the first projection is a faithfully flat morphism with all fibers isomorphic to the affine plane $\mathbb{A}^{2}$
over the corresponding residue fields, except for $\pi^{-1}(0)$ which is isomorphic to the cylinder $D\times \mathrm{Spec}(k[z])$ over the singular plane curve $D=\{y^r+t^s=0\}$.
It then follows for instance from \cite{Ka02} that $X(m,r,s,q)$ is not isomorphic to $\mathbb{A}^{3}$. 

Let us now pass to the proof of Theorem \ref{thm:family} itself. 
The first assertion is a particular case of \cite[Theorem 4.2]{DMJ14}
which is stated over the field of complex numbers but whose proof
remains valid over any field of characteristic zero. The second assertion
follows from a similar argument as in the proof of \cite[Theorem 1.3]{DMJ11}.
Writing $X_{q}=X(m,r,s,q)$, it is enough to show that $X_{q}\times\mathbb{A}^{1}$
is $k$-isomorphic to $X_{q(0)}\times\mathbb{A}^{1}$. Up to a linear
change of coordinate on the ambient space, we may assume that $q(0)=1$.
Then we let $f(x)\in k[x]$ be a polynomial such that $\exp(xf(x))\equiv q(x)$
modulo $x^{m}$, and we choose relatively prime polynomials $g_{1},g_{2}\in k[x]$
such that ${\displaystyle \exp(\frac{1}{r}xf(x))\equiv}g_{1}(x)$ and
${\displaystyle \exp(\frac{1}{s}xf(x))\equiv}g_{2}(x)$ modulo $x^{m}$.
Since $g_{1}(0)=g_{2}(0)=1$, the polynomials $x^{m}g_{1}(x)$, $x^{m}g_{2}(x)$
and $g_{1}(x)g_{2}(x)$ generate the unit ideal in $k[x]$. Therefore
we can find polynomials $h_{1},h_{2},h_{3}\in k[x]$ such that the
matrix 
\[
\left(\begin{array}{ccc}
g_{1}(x) & 0 & x^{m}\\
0 & g_{2}(x) & x^{m}\\
h_{1}(x) & h_{2}(x) & h_{3}(x)
\end{array}\right)
\]
belongs $\mathrm{GL}_{3}(k[x])$. This matrix hence defines
a $k[x]$-automorphism of $k[x][y,t,w]$, and a direct computation
shows that the latter maps the ideal $(x^{m},y^{r}+t^{s}+xq(x))$
onto the one $(x^{m},q(x)(y^{r}+t^{s}+x))=(x^{m},y^{r}+t^{s}+x)$,
where the equality follows from the fact that $x$ and $q(x)$ are
relatively prime. These ideals coincide respectively with the centers
of the affine birational morphisms $\sigma_{q}=\mathrm{pr}_{x,y,t,w}:X_{q}\times\mathrm{Spec}(k[w])\rightarrow\mathrm{Spec}(k[x,y,t,w])$
and $\sigma_{1}=\mathrm{pr}_{x,y,t,w}:X_{1}\times\mathrm{Spec}(k[w])\rightarrow\mathrm{Spec}(k[x,y,t,w])$
in the sense of \cite[Theorem 1.1]{KaZa99}, and it follows from the
universal property of affine modifications \cite[Proposition 2.1]{KaZa99}
that the corresponding $k$-automorphism of $\mathbb{A}^{4}=\mathrm{Spec}(k[x,y,t,w])$
lifts to an isomorphism between $X_{q}\times\mathbb{A}^{1}$ and $X_{1}\times\mathbb{A}^{1}$. 

As a corollary, we get the following result.

\begin{cor}
The threefolds $X(m,r,s,q)$ are all $\mathbb{A}^{1}$-contractible. 
\end{cor}

\begin{proof}
Since $X(m,r,s,q)$ is stably isomorphic to $X(m,r,s)$, it follows that they are actually isomorphic in the $\mathbb A^1$-homotopy category. By Theorem \ref{thm:KRcontractible}, the latter is contractible and it follows that the former is also contractible. 
\end{proof}
  
\section{Technical results}\label{sec:technical}

\subsection{\label{sub:weak-equiv} An explicit weak-equivalence}

We let $X(s)=\{x^{m}z=y^{r}+t^{s}+x\}\subset\mathbb{A}^{4}$, $m,r\geq2$ are fixed, $s\geq1$ and $(r,s)=1$, and
we let $L=\{x=y=t=0\}\subset X(s)$. Note that $X(1)$ is isomorphic
to $\mathbb{A}^{3}=\mathrm{Spec}(k[x,y,z])$ and that $X(1)\setminus L\cong\mathrm{Spec}(k[x,y,z])\setminus\{x=y=0\}\cong(\mathbb{A}^{2}\setminus\{0\})\times\mathrm{Spec}(k[z])$.
Therefore $X(1)\setminus L$ is weakly-equivalent to $\mathbb{A}^{2}\setminus \{0\}$ and the following
proposition provides in turn by composition a weak-equivalence $X(s)\setminus L\simeq\mathbb{A}^{2}\setminus\{0\}$
for every $s\geq2$. 
\begin{prop}
\label{prop:explicit}For every $s\geq2$, there exists a smooth quasi-affine
fourfold $W$ which is simultaneously the total space of Zariski locally
trivial $\mathbb{A}^{1}$-bundles $p_{s}:W\rightarrow X(s)\setminus L$ and $p_{1}:W\rightarrow X(1)\setminus L$. 
\end{prop}
To derive the existence of $W$, we observe that 
\[q=\mathrm{pr}_{x,t}\mid_{X(s)}:X(s)\setminus L\rightarrow\mathbb{A}^{2}\setminus\{0\}=\mathrm{Spec}(k[x,t])\setminus\{0\}\]
is a faithfully flat morphism restricting to a trivial $\mathbb{A}^{1}$-bundle $\mathrm{Spec}(k[x^{\pm1},t])\times\mathrm{Spec}(k[y])$
over the principal affine open subset $U_{x}=\{x\neq0\}$ of $\mathbb{A}^{2}\setminus\{0\}$.
On the other hand, the fiber of $q$ over the punctured line $C_{0}=\{x=0\}\simeq\mathrm{Spec}(k[t^{\pm1}])$
is isomorphic to the cylinder $C_{1}\times\mathrm{Spec}(k[z])$ over
the finite \'etale cover $h_{0}:C_{1}=\mathrm{Spec}(k[t^{\pm1}][y]/(y^{r}+t^{s}))\rightarrow C_{0}$
of $C_{0}$. This indicates roughly that $q:X(s)\setminus L\rightarrow\mathbb{A}^{2}\setminus\{0\}$
factors through an \'etale locally trivial $\mathbb{A}^{1}$-bundle
$\rho:X(s)\setminus L\rightarrow\mathfrak{S}$ over a smooth algebraic
space $\delta:\mathfrak{S}\rightarrow\mathbb{A}^{2}\setminus\{0\}$, obtained
from $\mathbb{A}^{2}\setminus\{0\}$ by ``replacing $C_{0}$ by $C_{1}$''.
The precise construction of $\mathfrak{S}$ given in the proof of
Lemma \ref{lem:quotient-space} below reveals that the isomorphy class
of $\mathfrak{S}$ depends only on $r$, in particular it depends neither on $m$ nor on $s$.

As a consequence, all the quasi-affine threefolds $X(s)\setminus L$,
$s\geq1$, have the structure of \'etale locally trivial $\mathbb{A}^{1}$-bundles
over a same algebraic space $\mathfrak{S}$. It follows that for every
$s\geq2$, the fiber product $W=(X(s)\setminus L)\times_{\mathfrak{S}}(X(1)\setminus L)$
is an algebraic space which is simultaneously the total space of an
\'etale locally trivial $\mathbb{A}^{1}$-bundle over $X(s)\setminus L$ and $X(1)\setminus L$ via the
first and second projections $p_{s}:W\rightarrow X(s)\setminus L$ and $p_{1}:W\rightarrow X(1)\setminus L$
respectively. Since the structure morphism of such a bundle is affine,
$W$ is actually a quasi-affine scheme and the local triviality of
$p_{s}$ and $p_{1}$ in the Zariski topology is an immediate consequence
of the fact that the group $\mathrm{Aut}(\mathbb{A}^{1})=\mathbb{G}_{m}\ltimes\mathbb{G}_{a}$
is special \cite{Gr58}. 
\begin{lem}
\label{lem:quotient-space} There exists a smooth algebraic space
$\delta:\mathfrak{S}\rightarrow\mathbb{A}^{2}\setminus\{0\}$ such that
for every $s\geq1$, the morphism $q:X(s)\setminus L\rightarrow\mathbb{A}^{2}\setminus\{0\}$
factors through an \'etale locally trivial $\mathbb{A}^{1}$-bundle
$\rho:X(s)\setminus L\rightarrow\mathfrak{S}$. \end{lem}
\begin{proof}
The quasi-affine threefold $X(s)\setminus L$ is covered by the two
principal affine open subsets $V_{x}=\{x\neq0\}$ and $V_{t}=\{t\neq0\}$.
Since $q\mid_{V_{x}}:V_{x}\rightarrow U_{x}=\mathrm{Spec}(k[x^{\pm1},t])$
is already a trivial $\mathbb{A}^{1}$-bundle as observed above, it
is enough to prove the existence of an algebraic space $\delta_{t}:\mathfrak{S}_{t}\rightarrow U_{t}=\mathrm{Spec}(k[x,t^{\pm1}])$
such that $q\mid_{V_{t}}:V_{t}\rightarrow U_{t}$ factors through an
\'etale locally trivial $\mathbb{A}^{1}$-bundle $V_{t}\rightarrow\mathfrak{S}_{t}$
and such that $\delta_{t}$ restricts to an isomorphism over $U_{t}\cap U_{x}$.
The desired algebraic space $\delta:\mathfrak{S}\rightarrow\mathbb{A}^{2}\setminus\{0\}$
will then be obtained by gluing $U_{x}$ and $\mathfrak{S}_{t}$ by
the identity along the open subsets $U_{x}\cap U_{t}$ and $\delta_{t}^{-1}(U_{t}\cap U_{x})\cong U_{t}\cap U_{x}$.

The algebraic space $\mathfrak{S}_{t}$ is constructed in the
form of a surface with an $r$-fold curve in the sense
of \cite[$\S$ 1.1]{DF14} as follows. First we let $h:C=\mathrm{Spec}(R)\rightarrow C_{0}$
be the Galois closure of the finite \'etale morphism $h_{0}:C_{1}=\mathrm{Spec}(k[t^{\pm1}][y]/(y^{r}+t^{s}))\rightarrow C_{0}=\mathrm{Spec}(k[t^{\pm1}])$,
that is, $C$ is the normalization of $C_{1}$ in the Galois closure $\kappa$
of the field extension $k(t)\hookrightarrow k(t)[y]/(y^{r}+t^{s})$.
The so defined field extension $\kappa$ is obtained from $k(t)$
by adding an $r$-th root of $t^{s}$, hence equivalently of $t$ since
$r$ and $s$ are relatively prime, and all $r$-th roots of $-1$.
In particular, neither $\kappa$ nor the curves $C_{1}$ and $C$
depend on $s$. By construction, $h:C\rightarrow C_{0}$ is an \'etale
torsor under the Galois group $G=\mathrm{Gal}(\kappa/k(t))$ which
factors as $h:C\stackrel{h_{1}}{\rightarrow}C_{1}\stackrel{h_{0}}{\rightarrow}C_{0}$
where $h_{1}:C\rightarrow C_{1}$ as an \'etale torsor under a certain
subgroup $H$ of $G$ of index $r$. 

The polynomial $y^{r}+t^{s}\in R[y]$ splits as $y^{r}+t^{s}=\prod_{\overline{g}\in G/H}(y-\lambda_{\overline{g}})$
for some elements $\lambda_{\overline{g}}\in R$, $\overline{g}\in G/H$
on which the Galois group $G$ acts transitively by $g'\cdot\lambda_{\overline{g}}=\lambda_{\overline{(g')^{-1}\cdot g}}$.
Furthermore, since $h_{0}:C_{1}\rightarrow C_{0}$ is \'etale, it
follows that for distinct $\overline{g},\overline{g}'\in G/H$, $\lambda_{\overline{g}}-\lambda_{\overline{g}'}\in R$
is an invertible regular function on $C$. This implies in turn that
there exists a collection of elements $\sigma_{\overline{g}}(x)\in B=R[x]$
with respective residue classes $\lambda_{\overline{g}}\in R=B/xB$
modulo $x$ on which $G$ acts by $g'\cdot\sigma_{\overline{g}}(x)=\sigma_{\overline{(g')^{-1}\cdot g}}(x)$,
and a $G$-invariant polynomial $s(x,y)\in B[y]$ such that in $B[y]$
one can write 
\[
y^{r}+t^{s}+x=\prod_{\overline{g}\in G/H}(y-\sigma_{\overline{g}}(x))+x^{m}s(x,y).
\]
It follows that $\tilde{V}_{t}=V_{t}\times_{U_{t}}\mathrm{Spec}(B)$
is isomorphic to the closed sub-variety of $\mathrm{Spec}(B[y,z_{1}])$
defined by the equation 
\[
x^{m}z_{1}=\prod_{\overline{g}\in G/H}(y-\sigma_{\overline{g}}(x))
\]
where $z_{1}=z-s(x,y)$. Since for distinct $\overline{g},\overline{g}'\in G/H$,
$\lambda_{\overline{g}}-\lambda_{\overline{g}'}$ is invertible, the
closed sub-scheme $\{x=0\}\subset\tilde{V}_{t}$ is the disjoint union
of $r$ closed sub-schemes $D_{\overline{g}}\cong\mathrm{Spec}(R[z_{1}])$
with respective defining ideals $(x,y-\sigma_{\overline{g}}(x))\in\Gamma(\tilde{V}_{t},\mathcal{O}_{\tilde{V}_{t}})$,
on which $G$ acts by permutation. The variety $\tilde{V}_{t}$ is
covered by the affine open subsets $\tilde{V}_{t,\overline{g}}=\tilde{V}_{t}\setminus\bigcup_{\overline{g}'\in (G/H)\setminus\{\overline{g}\}}D_{\overline{g}'}$,
$\overline{g}\in G/H$, and one checks using the above expression
for $y^{r}+t^{s}+x$ that the rational map 
\[
\tilde{V}_{t,\overline{g}}\dashrightarrow\mathrm{Spec}(B[u_{\overline{g}}]),\,(x,y,z_{1})\mapsto(x,\frac{y-\sigma_{\overline{g}}(x)}{x^{m}}=\frac{z}{\prod_{\overline{g}'\in (G/H) \setminus\{\overline{g}\}}(y-\sigma_{\overline{g}'}(x))})
\]
is an isomorphism of schemes over $\mathrm{Spec}(B)$. Altogether,
this implies that the faithfully flat morphism $\tilde{q}:\tilde{V}_{t}=V_t\times_{U_t} \mathrm{Spec}(B)\rightarrow\mathrm{Spec}(B)$
factors through a Zariski locally trivial $\mathbb{A}^{1}$-bundle
$\tilde{\rho}_{t}:\tilde{V}_{t}\rightarrow S$ over the scheme $\tilde{\delta}:S\rightarrow\mathrm{Spec}(B)$
obtained by gluing $r$ copies $S_{\overline{g}}$, $\overline{g}\in G/H$,
of $\mathrm{Spec}(B)\cong C\times\mathbb{A}^{1}$ by the identity
along the principal open subsets $\mathrm{Spec}(B_{x})\simeq C\times\mathrm{Spec}(k[x^{\pm1}])\subset S_{\overline{g}}$.
More precisely, $\tilde{\rho}_{t}:\tilde{V}_{t}\rightarrow S$ is
a Zariski locally trivial $\mathbb{A}^{1}$-bundle with local trivializations
$\tilde{V}_{t}\mid_{S_{\overline{g}}}\cong\mathrm{Spec}(B[u_{\overline{g}}])$
and transition isomorphisms over $S_{\overline{g}}\cap S_{\overline{g'}}\cong \mathrm{Spec}(B_x)$
of the form $u_{\overline{g}}\mapsto u_{\overline{g}'}=u_{\overline{g}}+x^{-m}(\sigma_{\overline{g}}(x)-\sigma_{\overline{g}'}(x))$.
The action of $G$ on $\tilde{V}_{t}$ descends to a fixed point free
action on $S$ defined locally by $S_{\overline{g}}\ni(c,x)\mapsto(g'\cdot c,x)\in S_{(\overline{g')^{-1}\cdot g}}$.
A geometric quotient for the action of $G$ on $S$ exists in the
category of algebraic spaces in the form of an \'etale $G$-torsor
$S\rightarrow\mathfrak{S}_{t}:=S/G$ and, by construction, we obtain
a cartesian square \[\xymatrix{\tilde{V}_t \ar[d]_{\tilde{\rho}_t} \ar[r] & V_t\cong \tilde{V}_t/G \ar[d]^{\rho_t} \\ S \ar[r] & \mathfrak{S}=S/G }\] 
where the horizontal morphisms are \'etale $G$-torsors. The induced
morphism $\rho_{t}:V_{t}\rightarrow\mathfrak{S}_{t}$ is thus an \'etale
locally trivial $\mathbb{A}^{1}$-bundle. To complete the proof, it
remains to observe that by construction, the $G$-invariant morphism
$\mathrm{pr}_{1}\circ\tilde{\delta}:S\rightarrow\mathrm{Spec}(B)\cong U_{t}\times_{C_{0}}C\rightarrow U_{t}$
descends to a morphism $\delta_{t}:\mathfrak{S}_{t}\rightarrow U_{t}$
restricting to an isomorphism outside $\{x=0\}\subset U_{t}$. Note
that on the other hand, $\delta_{t}^{-1}(\{x=0\})$ is isomorphic
to the quotient of $C\times G/H$ by the diagonal action of $G$,
hence to $C/H\cong C_{1}$. \end{proof}

\begin{rem}
The variety $X(s)$ admits an action of the additive group $\mathbb{G}_{a}$,
generated by the locally nilpotent derivation $\partial=x^{m}\partial_{y}-ry^{r-1}\partial_{z}$
of its coordinate ring. The line $L=\{x=y=t=0\}\subset X(s)$ is precisely
the fixed point locus of this action, and the $\mathbb{A}^{1}$-bundle
$\rho:X(s)\setminus L\rightarrow\mathfrak{S}$ constructed in Lemma
\ref{lem:quotient-space} coincides in fact with the geometric quotient
$X(s)\setminus L\rightarrow(X(s)\setminus L)/\mathbb{G}_{a}$ taken
in the category of algebraic spaces. In particular, $\rho:X(s)\setminus L\rightarrow\mathfrak{S}$
is an \'etale $\mathbb{G}_{a}$-torsor for every $s\geq1$, and so
are the bundles $p_{s}:W\rightarrow X(s)\setminus L$ and $p_{1}:W\rightarrow X(1)\setminus L$
deduced by taking fiber products over $\mathfrak{S}$. 
\end{rem}

The following example illustrates in the particular case of the Russell cubic $X=\{x^2z=y^2-t^3+x\}$ over an algebraically closed field $k$ an alternative construction of a weak-equivalence $X\setminus L \simeq \mathbb{A}^2\setminus \{0\}$.

\begin{example} As explained above, the morphism $q:X\setminus L \rightarrow \mathbb{A}^2\setminus \{0\}$ induced by the restriction of the projection $\mathrm{pr}_{x,t}$ factors through an \'etale $\mathbb{G}_a$-torsor $\rho:X\setminus L\rightarrow \mathfrak{S}$ over the algebraic space $\delta:\mathfrak{S}\rightarrow \mathbb{A}^2\setminus \{0\}$ obtained by replacing the punctured line $C_0=\{x=0\}\subset \mathbb{A}^2$ by its nontrivial \'etale double cover $h_{0}:C_{1}=\mathrm{Spec}(k[t^{\pm1}][y]/(y^{2}-t^{3}))\rightarrow C_{0}$. Now consider the smooth affine threefold $V$ in $\mathbb{A}^4=\mathrm{Spec}(k[x,t,u,v])$ defined by the equation $v^2t-x^2u=1$. The restriction to $V$ of the projection $\mathrm{pr}_{x,v}$ is a Zariski locally trivial $\mathbb{A}^1$-bundle $V\rightarrow \mathbb{A}^2\setminus \{0\}$, providing a weak-equivalence $V\simeq \mathbb{A}^2\setminus \{0\}$. On the other hand, a similar computation as in the proof of Lemma \ref{lem:quotient-space} shows that 
the restriction to $V$ of the projection $\mathrm{pr}_{x,t}$ factors through an \'etale locally trivial $\mathbb{A}^1$-bundle $\eta:V\rightarrow \mathfrak{S}$, in fact again an \'etale $\mathbb{G}_a$-torsor corresponding to the geometric quotient of the fixed point free $\mathbb{G}_a$-action on $V$ generated by the locally nilpotent derivation $\partial=2tv\partial_u-x^2\partial_v$ of its coordinate ring. 

The fiber product $W=(X\setminus L)\times_{\mathfrak{S}} V$ is thus simultaneously the total space of a $\mathbb{G}_a$-torsor over $X\setminus L$ and $V$ via the first and the second projections respectively. Furthermore, since $V$ is affine,  $\mathrm{pr}_V:W\rightarrow V$ is a trivial $\mathbb{G}_a$-torsor over $V$. One can check that via these isomorphisms, $X\setminus L$ coincides with the geometric quotient of the fixed point free $\mathbb{G}_a$-action on $W\cong V\times \mathrm{Spec}(k[w])$ generated by the locally nilpotent derivation $\tilde{\partial}=\partial+(t^{2}-\displaystyle{\frac{3}{2}xu^{2}})\partial_w$, the corresponding quotient map $V\times \mathbb{A}^1\rightarrow X\setminus L$ of schemes over $\mathbb{A}^2=\mathrm{Spec}(k[x,t])$ being given by 
\[(u,v)\mapsto (y,z)=(x^{2}w+t^{2}u-\frac{1}{2}xu^{3},x^{2}w^{2}+\frac{1}{4}u^{6}+2t^{2}uw-xu^{3}w+t^{3}v-x^{3}v^{2}+2xtuv^{2}).\]
\end{example} 

\subsection{The very weak five lemma}\label{sub:weak}

In this section, we prove a variant of the five lemma in the framework of pointed model categories. For these notions, we refer to \cite{Hovey99}.
Let then $\mathcal C$ be a pointed model category. For context, recall that given a commutative diagram
\[
\xymatrix{A\ar[r]^-u\ar[d]^-f  & B\ar[r]^-v\ar[d]^-g & C\ar[d]^-h \\
A^\prime\ar[r]_-{u^\prime} & B^\prime\ar[r]_-{v^\prime} & C^\prime}
\]
where the rows are cofiber sequences, then $h$ is a weak-equivalence provided $f$ and $g$ are. In general, one can not deduce that either $f$ or $g$ is a weak-equivalence provided the other two vertical morphisms are weak-equivalences. However, one can prove the following particular case.

\begin{lem}[Very weak five lemma]\label{lem:vwfl}
Let  \[ \xymatrix{A\ar[r]^-u\ar[d]^-f & B\ar[r]^-v\ar[d]^-g & C\ar[d]^-h \\ A^\prime\ar[r]_-{u^\prime} & B^\prime\ar[r]_-{v^\prime} & C^\prime} \] be a commutative diagram in a pointed model category $\mathcal C$ such that the rows are cofiber sequences. Suppose that $f$ and $h$ are weak-equivalences and that $B$ is contractible (i.e. the map to the final object is a weak-equivalence). Then $g$ is also a weak-equivalence and $B^\prime$ is contractible.
\end{lem}

\begin{proof} 
It suffices to show that $[B^\prime,X]=*$ for any object $X$ of $\mathcal C$. Applying the functor $[\_,X]$ for any object $X\in \mathcal C$ to the above diagram, we get a commutative diagram of long exact sequences (of pointed sets and groups). We conclude that $[B^\prime,X]=*$ by a simple diagram chase. 
\end{proof}

\subsection{The Milnor-Witt $K$-theory sheaf}\label{sub:MW}

Recall from \cite[\S 3.1]{Morel08} that given a field $F$, one can define the Milnor-Witt $K$-theory of $F$ denoted by $K_*^{\mathrm{MW}}(F)$, which is a $\Z$-graded ring with explicit generators and relations given in \cite[Definition 3.1]{Morel08}. The relevant features for us are that $K_0^{\mathrm{MW}}(F)=GW(F)$, the Grothendieck-Witt ring of symmetric bilinear forms (as usual, we denote $\langle a\rangle$ the class of the symmetric bilinear form $(x,y)\mapsto axy$ for $a\in F^\times$) and that $K_1^{\mathrm{MW}}(L)$ is generated by symbols $[a]$ with $a\in F^\times$. Given elements $a_1,\ldots,a_n\in F^\times$, we denote by $[a_1,\ldots,a_n]$ the $n$-fold product $[a_1]\cdot \ldots\cdot [a_n]$. 

Given a discrete valuation $v:L\to \Z$ with valuation ring $\mathcal O_v$, uniformizing parameter $\pi_v$ and residue field $k(v)$ there is a unique homomorphism $\partial_v:K^{\mathrm{MW}}_*(F)\to K^{\mathrm{MW}}_*(k(v))$ of degree $-1$ satisfying the formula $\partial_v([\pi_v,a_1,\ldots,a_n])=[\overline a_1,\ldots,\overline a_n]$ and $\partial_v([a_1,\ldots,a_n])=0$ for $a_i\in \OO_v^\times$ (\cite[\S 3.2]{Morel08}). The problem with this residue homomorphism is that it depends on the choice of the uniformizing parameter $\pi_v$. This led to considering \emph{twisted} Milnor-Witt $K$-theory groups as follows.  

Let $V$ be a rank one $F$-vector space and let $V^0$ be the set of nonzero elements in $V$. It has a transitive action of $F^\times$ and we can see the free abelian group $\Z[V^0]$ as a $\Z[F^\times]$-module. On the other hand, there is an action of $F^\times$ on $K_n^{\mathrm{MW}}(F)$ for any $n\in\Z$ by multiplication by the form $\langle a\rangle\in K_0^{\mathrm{MW}}(F)$, and thus the groups $K_n^{\mathrm{MW}}(F)$ are also $\Z[F^\times]$-modules. Set $K_n^{\mathrm{MW}}(F,V):=K_n^{\mathrm{MW}}(F)\otimes_{\Z[F^\times]}\Z[V^0]$. 

The residue homomorphism allows to define for any $n\in \Z$ a (Nisnevich) sheaf on the category of smooth $k$-schemes $\KMW_n$. If $X$ is a smooth scheme, this sheaf has an explicit flasque resolution whose term in degree $i$ is of the form
\[
\displaystyle{\bigoplus_{x\in X^{(i)}} \KMW_{n-i}(k(x),\wedge^i \mathfrak m_x/\mathfrak m_x^2) }
\]
where $\mathfrak m_x$ is the maximal ideal in $\OO_{X,x}$. The boundary homomorphism $d$, built on the residue map defined above, is described in \cite[Definition 5.11]{Morel08}. 

This sheaf is well-behaved if the base field $k$ is infinite perfect. For the needs of this paper, this means that $\KMW_n$ is a strictly $\A^1$-invariant sheaf in the sense of \cite[Definition 7]{Morel08}. This follows from \cite[Corollary 5.43, Theorem 5.38]{Morel08} and consequently there exists a space $K(\KMW_n,i)$ in the motivic homotopy category $\mathcal H(k)$ such that $[X,K(\KMW_n,i)]_{\aone}=H^i_{Nis}(X,\KMW_n)=H^i_{Zar}(X,\KMW_n)$. All the functoriality properties we use derive from this result.

Another useful fact (and indeed their fundamental feature) about Milnor-Witt $K$-theory sheaves is that they allow to describe the first nontrivial homotopy sheaf of the motivic spheres $\A^n\setminus \{0\}$ for $n\geq 2$ (\cite[Theorem 6.40]{Morel08}). This allows to understand the endomorphism ring $[\A^n\setminus \{0\},\A^n\setminus \{0\}]_{\aone}$ for any $n\geq 2$. Indeed, it follows from \cite[Corollary 6.43]{Morel08} that $[\A^n\setminus \{0\},\A^n\setminus \{0\}]_{\aone}=\KMW_0(k)$ if $n\geq 2$. If $f:\A^n\setminus \{0\}\to\A^n\setminus \{0\}$ is a morphism in $\mathcal H(k)$, the class of $f$ in the previous endomorphism ring is called the (motivic) Brouwer degree of $f$. There are several concrete ways to compute the Brouwer degree, and we will use the following method below. It follows from \cite[\S 3.3]{Fasel08b} (or \cite[Corollary 4.5]{Asok12b}) that 
\[
[\A^n\setminus \{0\}, K(\KMW_n,n-1)]_{\aone}=H^n(\A^n\setminus \{0\},\KMW_n)=\KMW_0(k)\cdot \xi
\]
where $\xi\in H^n(\A^n\setminus \{0\},\KMW_n)$ is an explicit generator. Now $[\A^n\setminus \{0\},\A^n\setminus \{0\}]_{\aone}\simeq H^{n-1}(\A^n\setminus \{0\},\KMW_n)$ by \cite[Corollary 4.4]{Asok12b}, the isomorphism being given by $f\mapsto f^*(\xi)$. The Brouwer degree of $f$ is therefore the element $\alpha(f)\in \KMW_0(k)$ such that $f^*(\xi)=\alpha(f)\cdot \xi$. We have thus obtained the following lemma.

\begin{lem}\label{lem:concreteBrouwer}
Let $f:\A^n\setminus \{0\}\to \A^n\setminus \{0\}$ be a morphism in $\mathcal H(k)$. Then $f$ is an isomorphism if and only if 
\[
f^*:H^{n-1}(\A^n\setminus \{0\},\KMW_n)\to H^{n-1}(\A^n\setminus \{0\},\KMW_n)
\] 
is an isomorphism.
\end{lem}

\subsection{Some computations}\label{sub:computations}

The stage being set, we come back to our main purpose: the proof of Theorem \ref{thm:KRcontractible}. As explained in Section \ref{sec:main}, we have a commutative diagram of cofiber sequences
\begin{equation}\label{eqn:crucial}
\xymatrix{\A^2\setminus \{0\}\ar[r]\ar[d]_-i & \A^2\ar[d]\ar[r] & (\pone)^{\wedge 2}\ar[d] \\ X\setminus L\ar[r] & X\ar[r] & X/(X\setminus L).} 
\end{equation} 
Now $L=\A^1$, and we have a weak-equivalence $X/(X\setminus L)\simeq L_+\wedge (\pone)^{\wedge 2}$ by homotopy purity (\cite[\S 3.2, Theorem 2.23]{Morel99}). Moreover, the morphism $(\pone)^{\wedge 2}\to L_+\wedge (\pone)^{\wedge 2}$ is the morphism induced by the inclusion $\{0\}\subset \A^1=L$. It follows that the right-hand vertical morphism in the above diagram is a weak-equivalence. In view of the very weak five lemma, we are therefore reduced to prove that $i:\A^2\setminus \{0\}\to X\setminus L$ is a weak-equivalence. We know from Section \ref{sub:weak-equiv} that there is an explicit weak-equivalence $g:X\setminus L\to \A^2\setminus \{0\}$ and consequently $H^{1}(X\setminus L,\KMW_2)=\KMW_0(k)\cdot \mu$ for some generator $\mu$, say $\mu=g^*(\xi)$. In view of Lemma \ref{lem:concreteBrouwer}, we are reduced to prove that 
\[
i^*:H^{1}(X\setminus L,\KMW_2)\to H^1(\A^2\setminus \{0\},\KMW_2)
\]
is an isomorphism to conclude. To view this, consider the commutative diagram
\[ 
\xymatrix{H^1(X\setminus L,\KMW_2)\ar[r]^-\partial\ar[d]_-{i^*} & H^2((\pone)^{\wedge 2},\KMW_2)\ar[r]\ar@{=}[d] & H^2(X,\KMW_2)\ar[d]\ar[r] & H^2(X\setminus Z,\KMW_2)\ar[d]^-{i^*} \\
H^1(\A^2\setminus \{0\},\KMW_2)\ar[r]_-{\partial^\prime} & H^2((\pone)^{\wedge 2},\KMW_2)\ar[r] & H^2(\A^2,\KMW_2)\ar[r] & H^1(\A^2\setminus \{0\},\KMW_2)} 
\] 
associated to Diagram (\ref{eqn:crucial}). The last two terms in the bottom row are trivial, and the last term in the top row is also trivial by \cite[Lemma 4.5]{Asok12b}. Moreover, $\partial^\prime$ is an isomorphism and it follows that the left-hand $i^*$ is an isomorphism if and only if $\partial$ is an isomorphism. Since the two left-hand groups in the top row are free $\KMW_0(k)$-modules of rank one and $\partial$ is $\KMW_0$-linear, we reduced the proof of Theorem \ref{thm:KRcontractible} to the following assertion.

\begin{prop}\label{prop:twoways}
The connecting homomorphism $\partial:H^{1}(X\setminus L,\KMW_2)\to H^2((\pone)^{\wedge 2},\KMW_2)$ is surjective.
\end{prop}

\begin{proof}[Lazy proof]
As the first row in the above diagram is exact, it is sufficient to prove that $H^2(X,\KMW_2)=0$. Now, it follows from the projective bundle theorem in \cite{Fasel09d} that 
\[
H^i(X,\KMW_j)=H^{i+n}(X_+\wedge (\pone)^{\wedge n},\KMW_{j+n})
\]
for any $i,n\in \N$ and any $j\in \Z$. Moreover, we get a split cofiber sequence 
\[ 
(\pone)^{\wedge n}\to X_+\wedge (\pone)^{\wedge n} \to X\wedge (\pone)^{\wedge n} 
\]
from \cite[Proposition 2.2.4]{Asok14}. As $X\wedge (\pone)^{\wedge n}=*$ if $n$ is big enough by the main result of \cite{Hoyois15}, we find $H^i(X,\KMW_j)=H^{i+n}((\pone)^{\wedge n},\KMW_{j+n})$ for $i\geq 1$ and the latter is trivial.
\end{proof}

\begin{proof}[Explicit proof]
The generator of $H^2((\pone)^{\wedge 2},\KMW_2)$ is explicitly given by the class of the cocycle
\[
\langle 1\rangle\otimes \overline t\wedge \overline y\in \KMW_0(k(L),\wedge^2\mathfrak m_L/\mathfrak m_L^2)
\]
and we show that it is the boundary of a cocycle in $H^{1}(X\setminus L,\KMW_2)$. 

With this in mind, consider the (integral) subvarieties $M:=\{y=0\}\subset X$, $N:=\{t=0\}\subset X$ and $L^\prime:=\{y=t=0; x^{m-1}z=1 \}\subset X$. Observe that $M\cap N=L\coprod L^\prime$. The element $[y]\otimes \overline t\in \KMW_1(k(N),\mathfrak m_N/\mathfrak m_N^2)$ has non trivial boundary only on $L$ and $L^\prime$ and their values are respectively $\langle 1\rangle \otimes \overline t\wedge \overline y\in\KMW_0(k(L),\wedge^2(\mathfrak m_L/\mathfrak m_L^2))$ and $\langle 1\rangle \otimes \overline t\wedge \overline y\in\KMW_0(k(L^\prime),\wedge^2(\mathfrak m_{L^\prime}/\mathfrak m_{L^\prime}^2))$. It follows that $[y]\otimes \overline t\in \KMW_1(k(N),\mathfrak m_N/\mathfrak m_N^2)$ is not a cocycle on $X\setminus L$, and we now modify it to obtain an actual cocycle. 

The symbol $[x^{m-1}z-1]\otimes \overline y\in \KMW_1(k(M),\mathfrak m_M/\mathfrak m_M^2)$ has non trivial boundary only on $L^\prime$, which we now compute. As $x(x^{m-1}z-1)=y^r+t^s$ and $x\not \in \mathfrak m_{L^\prime}$, we find $(x^{m-1}z-1)=x^{-1}(y^r+t^s)\in \mathcal O_{\mathfrak m_{L^\prime}}$ and it follows that the boundary of $[x^{m-1}z-1]\otimes \overline y$ is the same as the boundary of $[x^{-1}t^s]\otimes\overline y$. As $[x^{-1}t^s]=[x^{-1}]+\langle x^{-1}\rangle [t^s]=[x^{-1}]+\langle x\rangle [t^s]$ and $[t^s]=s_{\epsilon}[t]$ (see \cite[Lemma 3.14]{Morel08} for the definition of $s_\epsilon$), we finally find that the boundary of $[x^{m-1}z-1]\otimes \overline y$ is $\langle x\rangle s_{\epsilon}\otimes \overline y\wedge \overline t=\langle -x\rangle s_{\epsilon}\otimes \overline t\wedge \overline y$.

Let $S:=\{x^{m-1}z=1\}\subset X$. As $S$ and $M$ are different codimension $1$ subvarieties, we find \[ d([y,x^{m-1}z-1])=[x^{m-1}z-1]\otimes \overline y+\epsilon [y]\otimes \overline {x^{m-1}z-1}. \] As $d^2=0$ and $\epsilon=-\langle -1\rangle$, we find  \[ d([y]\otimes \overline {x^{m-1}z-1})=\langle -1\rangle d([x^{m-1}z-1]\otimes \overline y)=\langle x\rangle s_{\epsilon}\otimes \overline t\wedge \overline y. \] Now, $x$ is a unit on $S$ and it follows that $d(\langle x\rangle[y]\otimes \overline {x^{m-1}z-1})=s_{\epsilon}\otimes \overline t\wedge \overline y$. A similar computation shows that $d([x^{m-1}z-1]\otimes \overline t)=\langle x\rangle r_\epsilon \otimes \overline t\wedge \overline y$ and it follows that  
\[ 
d([t]\otimes \overline {x^{m-1}z-1})=\langle -1\rangle d([x^{m-1}z-1]\otimes \overline t)=\langle -x\rangle r_{\epsilon}\otimes \overline t\wedge \overline y. 
\] 
Thus $d(\langle -x\rangle[t]\otimes \overline {x^{m-1}z-1})=r_{\epsilon}\otimes \overline t\wedge \overline y$. As $(r,s)=1$, we may suppose (switching $r$ and $s$ if necessary) that there exists $g,h\in \N$ such that $gr-hs=1$. For any integers $p,q$, we have $p_\epsilon q_\epsilon=(pq)_\epsilon$ and it follows that $g_\epsilon r_\epsilon -h_\epsilon s_\epsilon=\langle \pm 1\rangle$ (more precisely, it is $\langle 1\rangle$ if $gr$ is odd and $\langle -1\rangle$ otherwise).

In short, we see that $d(g_\epsilon \langle -x\rangle[t]\otimes \overline {(x^{m-1}z-1)}-h_\epsilon \langle x\rangle[y]\otimes \overline {(x^{m-1}z-1)})=\langle \pm 1\rangle \otimes \overline t\wedge \overline y$ and consequently \[ 
[y]\otimes \overline t-\langle \pm 1\rangle (g_\epsilon \langle -x\rangle[t]\otimes \overline {(x^{m-1}z-1)}-h_\epsilon \langle x\rangle[y]\otimes \overline {(x^{m-1}z-1)}) 
\] 
is a cocycle mapping to the generator of $H^2((\pone)^{\wedge 2},\KMW_2)$ under the boundary map $\partial:H^1(X\setminus Z,\KMW_2)\to H^2((\pone)^{\wedge 2},\KMW_2)$. 
\end{proof}

%%%%%%%%%%%%%%%%%%%%%%%%%%%%%%%%%%%%%%%%%%%%%%%%%%%%%%%%%%%%%%%%%%%%%%%%%%%%%%%%%%%%%%%%%%%%%%%%%%%%%%%%%%%%%%%

\bibliographystyle{alpha}
\bibliography{Russel,Russel-2}

\end{document}